\documentclass[12pt]{amsart}
\usepackage{latexsym,amsfonts,amsthm,amsmath,amscd,amssymb}
\usepackage[all]{xy}
\usepackage{array}
\usepackage[pdftex]{hyperref}
\newtheorem{theorem}{Theorem}
\newtheorem{proposition}[theorem]{Proposition}
\newtheorem{lemma}[theorem]{Lemma}

\begin{document}

\title{Bernstein Type Results for Lagrangian Graphs with Partially
Harmonic Gauss Map}
\author{Wei Zhang}
\thanks{
Mathematics Classification Primary(2000): 53A10, 53A07, 53C38
\indent \ The author thanks his advisor Professor Dong for
suggesting him study this subject and some related methods. He is
also grateful to the referee who had pointed out a mistake in the
previous verison}
 \maketitle
 \begin{abstract}
We establish Bernstein Theorems for Lagrangian graphs which are
Hamiltonian minimal or have conformal Maslov form. Some known
results of minimal (Lagrangian) submanifolds are generalized.

 \end{abstract}

\section{Introduction}
The classical Bernstein theorem asserts any global graphic minimal
surface on $\mathbb{R}^2$ must be a plane. This result has been
generalized to $\mathbb{R}^n$ for $n\leq 7$, and higher dimensions
or codimesions under various growth conditions, see \cite{EH1990},
\cite{JX1999}, \cite{Wang_M_T2003a} and their references.

Due to string theory, minimal Lagrangian submanifolds, especially
the special Lagrangian submanifolds received much attention in
recent years (\cite{Joyce_D_D2001}). The Lagrangian graph in
$\mathbb{C}^n$ has the form of $(x, \nabla u(x))$, where $u(x)$ is a
smooth function on $\mathbb{R}^n$. Some authors established
Bernstein type results for minimal Lagrangian graph under various
conditions (\cite{JX2001}, \cite{Yuan_Y2002}, \cite{Wang_M_T2003a},
\cite{DHJ2006}). On the other hand, there are two nice
generalizations of minimal Lagrangian submanifolds: Hamiltonian
minimal Lagrangian submanifolds and Lagrangian submanifolds with
conformal Maslov form. They are introduced respectively by
\cite{Oh_Y_G1990} and \cite{RU1998}. It turns out that they have
partially harmonic Gauss maps into the Lagrangian Grassmannian, and
may be regarded as generalizations of submanifolds with parallel
mean curvature vector. In this paper, we will establish Bernstein
type results for these two classes of typical Lagrangian
submanifolds.

\begin{theorem}\label{Th:GenerYuan}
Let (x,$\nabla u$) represent a Hamiltonian minimal Lagrangian
submanifold $\Sigma$ in $\mathbb{C}^n$. If u is a smooth convex
function on $\mathbb{R}^n$, i.e.
$$Hess(u)\geq 0$$
then $\Sigma$ is an affine plane.
\end{theorem}

\begin{theorem}\label{Th:GenerHJW}
$\Sigma=(x,\nabla u)$ is a Lagrangian submanifold with conformal
Maslov form. Suppose that there exists a number $\beta_0$ with

\begin{equation*}
\beta_0<cos^{-n}(\frac{\pi}{2\sqrt{\kappa n}}), \quad \kappa=
\begin{cases}
1 & n=2\\
2 & n \geq 3
\end{cases}
\end{equation*}
such that
$$\Delta_u:=\sqrt{det(I+Hess^2(u))}\leq \beta_0$$
then $\Sigma$ is a plane.
\end{theorem}

Theorem \ref{Th:GenerYuan} was obtained in \cite{Yuan_Y2002} under
the condition that $\Sigma$ is minimal Lagrangian and there is some
kind of generalization of Yuan's Theorem in \cite{DHJ2006}; Theorem
\ref{Th:GenerHJW} was established in \cite{HJW1980} for the
submanifolds of parallel mean curvature in Euclidean space
$\mathbb{R}^N$ and improved in \cite{JX1999} later.

As we know, there is a fiberation of lagrangian grassmannian over
$S^1$ used to introduce the Maslov form. The basic ideas of this
paper is to project the Gauss map to $S^1$ or a convex domain in the
standard fiber. The partial harmonicity implies that these projected
map is harmonic. Therefore we may use the methods in \cite{HJW1980}
to prove our results.

\section{Preliminary}

Let $\mathcal{L}(n)$ be the Lagrangian Grassmannian consisting of
all oriented Lagrangian subspaces in $\mathbb{C}^n$. It can be
identified with $U(n)/SO(n)$ in a natural way. More over it is
totally geodesic in the usual oriented Grassmann manifold
$\mathcal{G}(n,n)$ (cf. \cite{JX2001}). We will explain this in a
more invariant way.

We define an automorphism $F_{\omega}$ of $\mathcal{G}(n,n)$, which
sends any $n$-subspace $P$ to its oriented symplectic orthogonal
with respect to the canonical symplectic form $\omega$ in
$\mathbb{C}^n$, i.e. $F_{\omega}(P)=J(P^{\bot})$, where $J$ the
complex structure. A subspace $P\in \mathcal{G}(n,n)$ is Lagrangian
if and only if $F_{\omega}(P)=P$ (cf. Lemma I.2.1 of
\cite{Audin_M}), so $\mathcal{L}(n)$ is totally geodesic in
$\mathcal{G}(n,n)$.

The determinant mapping $det: U(n) \rightarrow S^1$ descends to the
quotient by $SO(n)$, so that $U(n)/SO(n)$ is a bundle over $S^1$
with projection
$$
det: U(n)/SO(n)\rightarrow S^1
$$
and standard fiber $SU(n)/SO(n)$. Following proposition describes
some important gemetric properties of this fibration, which will be
used to prove the main results.

\begin{proposition}\label{Prop:RieProGra}
The projection $det: U(n)/SO(n) \rightarrow S^1$ is a Riemannian
submersion with totally geodesic fiber, and the horizontal lift of
$S^1$ is a geodesic orthogonal to each fibers. Furthermore, for any
open arc $c=\{ e^{i\theta}: \alpha<\theta<\beta ,
0<\beta-\alpha<2\pi\} $of $S^1$,  $(det)^{-1}(c)$ is isometric to a
Riemann product
$$(\frac{\alpha}{\sqrt{n}}, \frac{\beta}{\sqrt{n}}) \times SU(n)/SO(n)$$
 with metric $d\theta^2 \times h$, where h denotes the
standard metric on SU(n)/SO(n) as a symmetric space.
\end{proposition}

\begin{proof}
For convenience, from now on, always denote $O$ the point in
$\mathcal{L}(n)$ representing base plane $\mathbb{R}^n$.

First show the fiber is totally geodesic. Let
$F_\theta=det^{-1}(\theta)$ be the fiber at $e^{i\theta} \in S^1$,
where $\theta\in [0,2\pi)$. $\mathcal{L}(n)$ is symmetric, so it is
only necessary to check the fiber $F_0$. The tangent vectors to
$F_0$ at $O$ congruent to the form

\begin{equation*}
 \left(
  \begin{array}{ccc}
    \theta_1 &  &  \\
     & \ddots &  \\
     &  & \theta_n \\
  \end{array}
\right), \quad \sum \theta_k =0
\end{equation*}

The geodesics issue from these vectors are left cosets of diagonal
matrices:
$$
 \left[ \left(
  \begin{array}{ccc}
    e^{it\theta_1} &  &  \\
     & \ddots &  \\
     &  & e^{it\theta_n} \\
  \end{array}
\right)\right], \quad \sum \theta_k =0
$$
whose determinant is always 1, still contained in $F_0$. Hence $F_0$
is totally geodesic, then all the fibers are totally geodesic in
$\mathcal{L}(n)$.

We will directly show the property of local Riemannian product, and
Riemann submersion is an easy corolary.

There is a geodesic $\sigma (t)$ in the form of scalar matrix

$$
\left[ \left(
    \begin{array}{ccc}
      e^{\frac{it}{\sqrt{n}}} &  &  \\
         & \ddots &  \\
         &  & e^{\frac{it}{\sqrt{n}}} \\
    \end{array}
  \right) \right]
$$
where $t$ is the arc length parameter. This curve intersects with
it'self over period $2\pi\sqrt{n}$, i.e. a close geodesic of length
$2\pi \sqrt{n}$, topological $n$-cover of the $S^1$. Its tangent
vector at $O$ is:
$$
 \left(
  \begin{array}{ccc}
    \frac{1}{\sqrt{n}} &  &  \\
     & \ddots &  \\
     &  & \frac{1}{\sqrt{n}} \\
  \end{array}
\right)
$$

The inner product in $T \mathcal{L}(n)$ is given by $trAB^{*}$,
where $A$, $B$ are matrix represent the tangent vectors. Therefore
this curve is orthonormal to $F_0$ and by symmetric, to every fiber.

Since the scalar matrices commute with any other matrices, the
multiplication of scalar matrices descendent to $\mathcal{L}(n)$ and
is an isometric transformation between fibers. Thus when
$\beta<2\pi$, $(det)^{-1}(0, \beta)$ is isomorphic to $F_0 \cdot
\sigma(0, \frac{\beta}{\sqrt{n}}) \cong (0, \frac{\beta}{\sqrt{n}})
\times SU(n)/SO(n)$. The effect of act by $\sigma(t)$ is rotating
the planes by a given angle.

Consequently, the projection $det$ is a Riemann submersion with
constant dilation factor $\sqrt{n}$. $\sigma$ is one of the
horizontal lift of $S^1$.
\end{proof}

Now we introduce the concepts of Hamiltonian minimal Lagrangian
submanifolds and Lagrangian submanifolds with conformal Maslov form,
then state several properties close related to harmonic map.

Let $\eta$ be a normal vector field along the Lagrangian
submanifold $\Sigma$. Denote by $\alpha_\eta$ the 1-form on
$\Sigma$ defined by
$$
\alpha_\eta(X)=\omega(\eta, X), X \in T\Sigma
$$
where $\omega$ is the symplectic structure of $\mathbb{C}^n$.
Particularly, if $H$ is the mean curvature vector of $\Sigma$,
$\alpha_H$ is called the Maslov form.

The normal vector field $\eta$ is called Hamiltonian if the 1-form
$\alpha_\eta$ associated with $\eta$ is a exact form, i.e.
$\alpha_\eta=df$, where $f \in C^\infty(\Sigma)$.

According to \cite{Oh_Y_G1990}, a Lagrangian submanifold $\Sigma$ is
called Hamiltonian minimal ($H$-minimal for short) if it is a
critical point of the volume functional with respect to all
Hamiltonian variations.

\begin{lemma}\cite{Oh_Y_G1990}
The Lagrangian submanifold $\Sigma$ is H-minimal if and only if
$\delta \alpha_H=0$, i.e. the Maslov form is coclosed.

\end{lemma}

There is a holomorphic volume form $dz=dz_1\wedge dz_2\cdots \wedge
dz_n$ in $\mathbb{C}^n$, evaluating $dz$ on $\Sigma$'s tangent space
has scalar value $\gamma$. $\psi$ is called the Lagrangian angle of
oriented Lagrangian submanifold $\Sigma$ if $\gamma= e^{i\psi}$(cf.
\cite{Wolfson_J1997}). $\psi$ is a well defined function takes value
in $2\pi \mathbb{R}/\mathbb{Z}$.

Wolfson had proved that $d\psi=\alpha_H$(Theorem 1.2 in
\cite{Wolfson_J1997}). By Lemma 4, $\Sigma$ is $H$-minimal if and
only if $\psi$ is a harmonic real value function on it. Because the
exponential map is a totally geodesic map from $\mathbb{R}$ to
$S^1$, by the composition law of harmonic map, $\gamma$ is a
harmonic map to $S^1$.

The Gauss map $\nu$ takes values in $\mathcal{L}(n)\cong
U(n)/SO(n)$. Obviously, $\gamma= det\circ \nu$, thus there is a
commute diagram

$$\xymatrix{
                &         \mathcal{L}(n) \ar[d]^{det}     \\
  \Sigma \ar[ur]^{\nu} \ar[r]_{\gamma} & S^1             }$$

By Proposition ~\ref{Prop:RieProGra}, $\gamma :\Sigma\rightarrow
S^1$ is harmonic if and only if $\tau(\nu)^H\equiv 0$, where
$\tau(\nu)^H$ is the horizontal component of the tension field
$\tau(\nu)$ with respect to the projection $det$. Thus $\Sigma$ is
$H$-minimal if and only if its Gauss map $\nu$ is horizontally
harmonic.

We had introduced the Maslov form $\alpha_H$. According to
\cite{RU1998}, a Lagrangian submanifold $\Sigma$ in $\mathbb{C}^n$
is said to have conformal Maslov form (CMF briefly), if $JH$ is a
conformal vector fields on $\Sigma$. The Gauss map of such manifolds
is vertically harmonic, i.e. $\tau(\nu)^V\equiv 0$, where
$\tau(\nu)^V$ is the vertical component of the tension field
$\tau(\nu)$ with respect to the projection $det$.

\section{Proof of the main theorems}
Simple Riemannian manifold, is quasi-isomorphic to Euclidean space,
where certain Liuville type theorem holds on. For instance,

\begin{theorem}[\cite{HJW1980}]\label{Th:LiouvilleType}
Let $\nu : \Sigma \mapsto M$  be a harmonic map, $\Sigma$ a simple
Riemannian manifold and M is Riemannian manifold whose sectional
curvature is bounded from above by a constant $\kappa \geq 0$.
Denote $B_R(q)$ a geodesic ball of radius
$R<\frac{\pi}{2\sqrt{\kappa}}$ which does not meet the cut locus of
$q$. If the range $\nu(\Sigma)$ is contained in $B_R(q)$, then $\nu$
is a constant map.
\end{theorem}

Apply this theorem to the Gauss map of a submanifold with parallel
mean curvature, \cite{HJW1980} got Bernstein type theorem.

Reader can find the precise definition of simple manifold in
\cite{HJW1980}. In the context of lagrangian graph, uniformly
bounded $Hess(u)$ assures the simpleness.

At the same time, so far as we know, the Lagrangian fiberation has
many local nice properties. The Gauss image $X$ of an oriented
Lagrangian graph $\Sigma$ lies in a special region of
$\mathcal{L}(n)$, which makes it possible to extend this local
properties to global. To do this ,we need a natural representation
of the $X$ in $\mathcal{L}(n)$, i.e. picking up an element in every
coset of $U(n)/SO(n)$ canonically, to represent the given tangent
plane.

\begin{lemma}
If $P$ is a tangent plane of $\Sigma=(x,\nabla u)$ in
$\mathbb{C}^n$, denote $\lambda_k$ the eigenvalues of $Hess(u)$,
then in its corresponding coset, there exists $V=S
diag(e^{i\theta_1}, \cdots, e^{i\theta_n}) S^{-1}$, where $S \in
SO(n)$, and $\theta_k=\arctan \lambda_k$ are the critical angles
between $P$ and $\mathbb{R}^n$.
\end{lemma}

\begin{proof}

This kind of $P$ in $\mathbb{C}^n \backsimeq \mathbb{R}^{2n}$,
contains no vectors orthogonal to $\mathbb{R}^n$, so we can define a
linear map form $\mathbb{R}^n$ to its compliment $J\mathbb{R}^n$:
$F_P(a)=b$, where $a\in \mathbb{R}^n$, $b\in J\mathbb{R}^n$, iff
there is a vector $v$ in $P$, s.t. ${\pi}_1 v=a$, ${\pi}_2 v=b$.
$F_P$ is $Hess(u)$ essentially. By eigenvalue decomposition, there
is a orthonormal basis $s_k$ of $\mathbb{R}^n$, s.t.
$F_P(s_k)=\lambda_k J(s_k)$. According to \cite{Wong_Y_C1967},
\cite{JX1999}, $\theta_k=\arctan\lambda_k$ are the critical angles.
Then $e_k=e^{i{\theta}_k} s_k$ are orthonomal basis of $P$. Let $S$
be the transformation sends $s_k$ to the standard basis of
$\mathbb{R}^n$. $\{s_k\}$ and the standard basis of $\mathbb{R}^n$
can be also viewed as complex basis of $\mathbb{C}^n$, while $S$ a
transform of complex space. Thus the matrix represent $P$ in the
form of:
$$
S \left(
\begin{array}{ccc}
    e^{i\theta_1} &  &   \\
       & \ddots &   \\
       &  & e^{i\theta_n} \\
  \end{array}
  \right)  S^{-1}
  $$
Since the graph is global oriented, we can arrange $s_k$ making
corresponding $e_k$ give the due orientation, so $S$ belongs to
$SO(n)$ rather than $O(n)$.

\end{proof}

Given any tangent plane $P$, corresponds to a unique set of
$\{\theta_k\}$ where $-\frac{\pi}{2}<\theta_k<\frac{\pi}{2}$. Write
$P_{\theta_1,\cdots,\theta_n}$ in stead of $P$.

 \subsection{Prove Theorem 1} \

This representation is unique up to permutations of $\theta_k$.
Define $\tilde{\gamma}=\sum\theta_k$, it is a well defined function
taking value in $\mathbb{R}$. The following diagram commutes:

$$
 \xymatrix{
                &         R \ar[d]^{e^{it}}     \\
  \Sigma \ar[ur]^{\tilde{\gamma}} \ar[r]_{\gamma} & S^1             }
$$
i.e. $\tilde{\gamma}$ is a representation of the Lagrangian angle
$\psi$ and $\tilde{\gamma}$ is harmonic iff $\gamma$ is.

From theorem \ref{Th:LiouvilleType}, any bounded harmonic function
on simple manifold must be constant. If $\Sigma$ is simple,
$\tilde{\gamma}=const$ is easily concluded form $\sum\theta_k\in
(-\frac{n\pi}{2}, \frac{n\pi}{2})$. Thus $\gamma=const$ either.
Recall Proposition 2.17 in \cite{HL1982}, which asserts that a
connected submanifold $\Sigma \in \mathbb{R}^{2n}=\mathbb{C}^n$ is
both Lagrangian and minimal if and only if $\Sigma$ is special
Lagrangian with respect to one of the calibration
$\mathrm{Re}\{e^{i\theta}dz\}$. If $\gamma$ is const,
$\mathrm{Re}\{e^{i\gamma}dz\}$ is the required calibration and
$\Sigma$ the SL submanifold. Therefore:

\begin{proposition}\label{Prop:GenerChern}

 $\Sigma=(x,\nabla u)$ is Hamiltonian minimal.
 If $Hess(u)$ is uniformly bounded, then $\Sigma$ is minimal in the usual sense.

\end{proposition}

Chern(\cite{Chern_S_S1965}) proved that any graphic hypersurface in
$\mathbb{R}^{n+1}$ with parallel mean curvature must be minimal.
proposition~\ref{Prop:GenerChern} generalizes Chern's theorem to
Lagrangian case. Noticing the the fact after a certain kind of
rotation(see \cite{Yuan_Y2002}), $Hess(u)\geq 0$ force $\Sigma$ to
be simple and Theorem 1.1 in \cite{Yuan_Y2002}
\begin{theorem}[\cite{Yuan_Y2002}]
Suppose $\Sigma=(x,\nabla u)$ is a minimal Lagrangian submanifold of
$\mathbb{C}^n$ and u is a smooth convex function on $\mathbb{R}^n$,
then $\Sigma$ is an affine plane.
\end{theorem}
theorem~\ref{Th:GenerYuan} follows.

  \subsection{Prove Theorem 2} \

For the representation is unique, $X \cap F_\theta$ is divided into
disjoint components $X_{l,\theta}=\{P_{\theta_1,\cdots,\theta_n}\in
X|\sum_k \theta_k=2l\pi+\theta, l\in \mathbb{Z},
-[\frac{n}{4}]-1\leq l \leq [\frac{n}{4}]\}$.

%$X_0, X_{2\pi}, \cdots, X_{2l\pi}$, $l \in \mathbb{Z}$,
%$X_{2l\pi}=\{P\in X|\sum \theta_k =2l\pi\}$. Every component is
%disjoint with each other. Thus $X$ has a global Riemann product
%structure, and

We can define a global projection $\pi:X \rightarrow F_0$ by $\pi
(P_{\theta_1,\cdots,\theta_n})=P_{\theta_1,\cdots,\theta_n} \cdot
\sigma(-\frac{\sum \theta_k}{\sqrt{n}})$. Intuitionally, pull back
all the components in all the fibers to the neighborhood of
$X_{0,0}$.

Thus the Gauss image $X$ looks like a tube around the close geodesic
$\sigma$ and has a global Riemann product structure.

\begin{proposition}\label{Prop:TubeConst}
$\Sigma$ is a simple Lagrangian graph with conformal Maslov form.
Denote $X$ the image of it's Gauss map and $T_R$ the region
$\{\sum_k(\theta_k-\frac{\sum_i \theta_i}{n})^2\leq R^2,
R<\frac{\pi}{2\sqrt{2}}\}$. If $X \subset T_R$, then $\Sigma$ is a
plane.

\end{proposition}

\begin{proof}

Project $X$ to $F_0$ via $\pi$, denote $D=\pi(X)$, there is

$$\Sigma \xrightarrow{\nu} X \xrightarrow{\pi} D \subset F_0$$

If $\nu$ is vertical harmonic, $\pi\circ \nu$ is a harmonic map to
$F_0$ by the composition formula.

The geodesic ball $B_R$ of radius $R<\frac{\pi}{2\sqrt{2}}$ is a
convex domain in $\mathcal{G}(n,n)$(cf.\cite{HJW1980}). For $F_0$ is
the totally geodesic fiber of $\mathcal{L}(n)$ and $\mathcal{L}(n)$
is totally geodesic in $\mathcal{G}(n,n)$, so is $F_0$ in $G(n,n)$.
Thus $F_0 \cap B_R$ is a convex domain in $F_0$ and $D \subset F_0
\cap B_R$.

Following the theorem \ref{Th:LiouvilleType}, the harmonic map
$\pi\circ \nu$ is constant. Without losing generality, assuming the
image is $O$. Then the image of $\nu$ is the cosets represented by
$e^{i\theta} Id$, i.e.

$$Hess(u)=(\tan \theta) Id$$

Then $\frac{\partial^2 u}{\partial x_i \partial x_j}=0$ when $i\neq
j$. thus $\frac{\partial u}{\partial x_i}$ is only function about
$x_i$, so does $\frac{\partial^2 u}{\partial x_i^2}$. But
$\frac{\partial^2 u}{\partial x_i^2}=\frac{\partial^2 u}{\partial
x_j^2}$ for any $x_i,x_j$. Both sides have to be constant, i.e.
$Hess(u)=cId$.

Therefore $u$ is an affine plane.

\end{proof}

The region $T_R$ is not a strict tube with constant diameter, for
when $\frac{\sum_i \theta_i}{n}$ approaches $\pm\frac{\pi}{2}$, the
diameter will tend to zero, but this makes no difference to our
proof.

\noindent $\mathbf{Remark}$. Proposition \ref{Prop:TubeConst} says
nothing but if the Gauss map lies in a sufficient small tabular
neighborhood of radius $\frac{\pi}{2\sqrt{2}}$ around the closed
geodesic, then it is plane. In \cite{JX1999}, after delicate study
of the structure about Grassmannian, Jost and Xin show there is a
bigger convex domain $B_G$. Hence above proposition still holds if
the image of $\nu \circ \pi$ lies in $F_0 \cap B_G$.

We are in the position to complete the proof of
Theorem~\ref{Th:GenerHJW}. Denote
$\triangle_u=(det(I+Hess^2(u)))^\frac{1}{2}$, then $\triangle_u\leq
cos^{-n}(\frac{R}{\sqrt{n}})$ implies $\Sigma$ is simple and $\sum
{\theta_k}^2\leq R^2$, moreover $\sum_k(\theta_k-\frac{\sum_i
\theta_i}{n})^2\leq R^2$. By Proposition ~\ref{Prop:TubeConst}, $u$
is affine plane. The special case $n=2$ follows form the fact the
fiber $SU(2)/SO(2)$ is isomorphic to $S^2$. The radius of $F_0$'s
convex domain is $\frac{\pi}{2}$ rather than
$\frac{\pi}{2\sqrt{2}}$.

%\bibliographystyle{plain}
%\bibliography{phdBib}

\vfill

\noindent Wei Zhang

\

\noindent School of Mathematical Sciences

\noindent Fudan University

\noindent Shanghai, 200433, P. R.China

\

\noindent Email address: 032018009@fudan.edu.cn

\end{document}